\documentclass{amsart}
\usepackage{amssymb,mathtools}
\usepackage[british]{babel}
\usepackage{enumitem}
\usepackage[latin1]{inputenc}

\usepackage{url}
\usepackage{stmaryrd}
\usepackage{appendix}
\usepackage{xcolor}
\usepackage{stackengine}

\newtheorem{theorem}{Theorem}
\numberwithin{theorem}{section}
\newtheorem{corollary}[theorem]{Corollary}
\newtheorem{lemma}[theorem]{Lemma}
\newtheorem{proposition}[theorem]{Proposition}
\newtheorem*{min-bad-arr-lem}{Minimal bad array lemma}

\theoremstyle{definition}
\newtheorem{definition}[theorem]{Definition}

\newtheorem*{claim}{Claim}

\newcommand\ledot{\mathrel{\ensurestackMath{%
  \stackengine{-.5ex}{\lessdot}{-}{U}{c}{F}{F}{S}}}}
\newcommand{\base}[1]{\textstyle\bigcup #1}
\newcommand{\barrsucc}[2]{#1 \ledot #2}

\newcommand{\barrpropsucc}[2]{#1 \lessdot #2}
\newcommand{\extbarr}[3]{E(#1,#2,#3)}

\newcommand{\minbarr}[2]{m(#1,#2)}

\title{The logical strength of minimal bad arrays}
\author{Anton Freund}
\author{Fedor Pakhomov}
\author{Giovanni Sold\`a}

\address{Anton Freund, University of W\"urzburg, Institute of Mathematics, Emil-Fischer-Stra{\ss}e~40, 97074 W\"urzburg, Germany}
\email{anton.freund@uni-wuerzburg.de}
\address{Fedor Pakhomov and Giovanni Sold\`a, Department of Mathematics: Analysis, Logic and Discrete Mathematics, Ghent University, Krijgslaan 281 S8, 9000 Ghent, Belgium}
\email{fedor.pakhomov@ugent.be{\normalfont, }giovanni.a.solda@gmail.com}

\thanks{The work of Anton Freund has been funded by the Deutsche Forschungsgemeinschaft (DFG, German Research Foundation) -- Project number 460597863. The work of Fedor Pakhomov and Giovanni Sold\`a has been funded by the FWO grant G0F8421N}

\begin{document}

\begin{abstract}
This paper studies logical aspects of the notion of better quasi order, which has been introduced by C.~Nash-Williams (\emph{Mathematical Proceedings of the Cambridge Philosophical Society} 1965 \& 1968). A central tool in the theory of better quasi orders is the minimal bad array lemma. We show that this lemma is exceptionally strong from the viewpoint of reverse mathematics, a framework from mathematical logic. Specifically, it is equivalent to the set existence principle of $\Pi^1_2$-comprehension, over the base theory~$\mathsf{ATR_0}$. 
\end{abstract}


\maketitle

\section{Introduction}

A quasi order~$Q$ with order relation~$\leq_Q$ is a well quasi order if any infinite sequence $q_0,q_1,\ldots$ in~$Q$ admits $i<j$ with~$q_i\leq_Q q_j$. There is a deep theory of well quasi orders~\cite{kruskal-rediscovered} with the graph minor theorem~\cite{robertson-seymour-gm} as a towering achievement. A limitation of well quasi orders is the lack of closure properties under infinitary operations. In particular, R.~Rado~\cite{rado-counterexample} has exhibited a well quasi order~$Q$ such that the power set $\mathcal P(Q)$ is no well quasi order when we put
\begin{equation*}
X\leq_{\mathcal P(Q)}Y\quad:\Leftrightarrow\quad\text{for any~$p\in X$ there is a~$q\in Y$ with $p\leq_Q q$}.
\end{equation*}
To secure such closure properties, C.~Nash-Williams~\cite{nash-williams-trees,nash-williams-bqo} has introduced the more restrictive notion of better quasi order, which we recall below. As a famous application, we mention R.~Laver's proof~\cite{laver71} of Fra\"iss\'e's conjecture that the embeddability relation on $\sigma$-scattered linear orders is a well and even a better quasi order. A~central tool in the theory of better quasi orders is the minimal bad array lemma, which is also known as the forerunning principle. In the present paper, we use methods from logic to show that any proof of that principle must use exceptionally strong set existence axioms. A precise formulation of this result is given below.

Let us introduce some notation that is needed to define the notion of better quasi order. The collections of finite and countably infinite subsets of a set~$V$ are denoted by $[V]^{<\omega}$ and~$[V]^\omega$, respectively. Sets from~$[\mathbb N]^{<\omega}\cup[\mathbb N]^\omega$ are identified with the sequences that enumerate them in increasing order. For such sequences we write $s\sqsubset t$ to assert that $s$ is a proper initial segment of~$t$. Note that this forces~$s$ to be finite. For $X\in[\mathbb N]^\omega$ with least element~$x$, we put $X^-:=X\backslash\{x\}$. In the following we use infix notation $p\,R\,q$ to express $(p,q)\in R$.

\begin{definition}
Consider a set~$Q$ with a binary relation~$R$. By a $Q$-array we mean a map $f:[V]^\omega\to Q$ with $V\in[\mathbb N]^\omega$ that is continuous in the sense that each~$X\in[V]^\omega$ admits an $s\sqsubset X$ such that $f$ is constant on~$\{Y\in[V]^\omega\,|\,s\sqsubset Y\}$. Such an $f$ is called $R$-bad if there is no~$X\in[V]^\omega$ with $f(X)\,R\,f(X^-)$. We say that $R$ is a better relation on~$Q$ if no $Q$-array is~$R$-bad. A better quasi order is a quasi order such that the order relation is a better relation on the underlying set.
\end{definition}

It is instructive to show that any better quasi order is a well quasi order. In~the proof of Proposition~\ref{prop:pouzet}, we will implicitly confirm that $Q\mapsto\mathcal P(Q)$ preserves better quasi orders. The case of better relations that are not quasi orders has not been studied widely, even though it is mentioned in~\cite{pouzet-better-relation,shelah-bqo}. We will consider it in an intermediate argument. Let us note that any better relation is reflexive.

The given definition of better quasi orders does essentially coincide with the original definition in terms of blocks. We recall that a block with base $V\in[\mathbb N]^\omega$ is a set $B\subseteq[V]^{<\omega}$ such that any $X\in[V]^\omega$ has a unique initial segment in~$B$, which may not be empty. The latter ensures that the base can be recovered as $\bigcup B$. Any function $f_0:B\to Q$ on a block $B$ with base~$V$ induces a $Q$-array $f:[V]^\omega\to Q$ given by $f(X)=f_0(s)$ for $s\in B$ with $s\sqsubset X$. Continuity entails that any array is induced in this way. In the following, we implicitly assume that all arrays are given as functions on blocks. This has the advantage that arrays become countable objects, which will be relevant for our logical analysis. Later, we will briefly discuss another equivalent definition of better quasi orders in terms of special blocks, which are called barriers. There is a further equivalent definition due to S.~Simpson~\cite{simpson-borel-bqos}, who admits all Borel measurable functions as arrays. In our logical analysis, we~will not consider this last definition by Simpson.

Now that we have discussed the definition of better quasi order, we introduce the notions that appear in the minimal bad array lemma.

\begin{definition}\label{def:minimal-bad}
A partial ranking of a reflexive binary relation~$R$ on a set~$Q$ is a well-founded partial order~$\leq'$ on~$Q$ such that $p\,R\,q\leq'r$ entails $p\,R\,r$. Given such a ranking and $Q$-arrays $f:[V]^\omega\to Q$ and $g:[W]^\omega\to Q$, we write $f\leq' g$ to express that we have $V\subseteq W$ as well as $f(X)\leq'g(X)$ for all~$X\in[V]^\omega$. In case we even have $f(X)<'g(X)$ for all $X\in[V]^\omega$, we write $f<'g$. A $Q$-array $g$ is said to be $\leq'$-minimal $R$-bad if it is $R$-bad and there is no $R$-bad $Q$-array $f<'g$.
\end{definition}

Let us note that $q\leq'r$ implies $q\,R\,r$ when $\leq'$ is a partial ranking of~$R$, given that the latter is reflexive. Conversely, if $R$ is transitive and $q\leq'r$ implies $q\,R\,r$, then $p\,R\,q\leq'r$ implies $p\,R\,r$. So in the case of quasi orders, our notion of partial ranking coincides with the usual one. The proofs of Proposition~\ref{prop:pouzet} and Theorem~\ref{thm:pi12-imp-mbal} suggest that the given definition is pertinent for relations that are not transitive. We can now present one of the central tools from the theory of better quasi orders.

\begin{min-bad-arr-lem}[`Simpson's version']
Consider a quasi order~$Q$ and a partial ranking~$\leq'$ of the order relation~$\leq_Q$ on~$Q$. Whenever~$f_0$ is a $\leq_Q$-bad $Q$-array, there is a $Q$-array~$f\leq'f_0$ that is $\leq'$-minimal $\leq_Q$-bad.
\end{min-bad-arr-lem}

The given version coincides with the one of Simpson~\cite{simpson-borel-bqos}, except that the latter works with a larger class of Borel measurable arrays, as mentioned above. A different definition of the relation $<'$ between arrays has been given by Laver~\cite{laver-min-array}. It leads to a different variant of the minimal bad array lemma, which we will later introduce as Laver's version. In both cases, we have a version for quasi orders and an obvious generalization to reflexive relations. We will see that our results apply to all indicated versions.

Our logical analysis takes place within the framework of reverse mathematics. For a comprehensive introduction, we refer to the founding paper by H.~Friedman~\cite{friedman-rm} and the textbook by S.~Simpson~\cite{simpson09}. To give a rather informal account of a central idea, we note that each class $\mathcal C$ of properties (more precisely: formulas of second order arithmetic) gives rise to an axiomatic principle called $\mathcal C$-comprehension, which asserts that the set $\{n\in\mathbb N\,|\,\varphi(n)\}$ exists for any~$\varphi$ from~$\mathcal C$. One writes $\mathcal C\textsf{-CA}_0$ for the theory that is axiomatized by $\mathcal C$-comprehension and some basic facts about the natural numbers. Important cases include the class $\Delta^0_1$ of properties that are algorithmically decidable, the class $\Pi^0_\infty$ of arithmetical properties (given by formulas that quantify over finite but not over infinite objects), the class $\Pi^1_1$ of coanalytic properties, and the class $\Pi^1_2$ of properties that have the form $\forall X\subseteq\mathbb N:\psi(n,X)$ for analytic~$\psi$. The theories $\Delta^0_1\textsf{-CA}_0$ and $\Pi^0_\infty\text{-CA}_0$ are usually denoted by $\mathsf{RCA_0}$ (recursive comprehension) and $\mathsf{ACA_0}$ (arithmetical comprehension), respectively. We will also encounter the theory $\mathsf{ATR_0}$ of arithmetic transfinite recursion, which permits to iterate arithmetical comprehension along arbitrary well orders.

To analyze a given theorem in reverse mathematics, one will often determine a class~$\mathcal C$ such that the principle of~$\mathcal C$-comprehension is equivalent to the theorem in question. Such an equivalence is informative if one proves it over a sufficiently weak base theory. A common choice is $\mathsf{RCA_0}$, but we will mostly work with $\mathsf{ATR_0}$ as base theory. This makes sense because Simpson's version of the minimal bad array lemma is typically used in conjunction with the open Ramsey theorem, which is equivalent to arithmetic transfinite recursion over~$\mathsf{RCA_0}$.

Now that both the mathematical notions and the logical framework have been fixed, we can formulate the main result of the present paper.

\begin{theorem}\label{thm:main}
The following are equivalent over~$\mathsf{ATR_0}$:
\begin{enumerate}[label=(\roman*)]
\item the principle of $\Pi^1_2$-comprehension,
\item Simpson's version of the minimal bad array lemma (as stated above),
\item Laver's version of the minimal bad array lemma (as stated in Section~\ref{sect:mba-from-Pi12CA}).
\end{enumerate}
\end{theorem}

Let us now put our result into context. In his textbook on reverse mathematics, Simpson~\cite{simpson09} states ``that a great many ordinary mathematical theorems are provable in $\mathsf{ACA_0}$, and that of the exceptions, most are provable in $\Pi^1_1\textsf{-CA}_0$." In~\cite{marcone-bad-sequence} it is shown that $\Pi^1_1$-comprehension is equivalent to the minimal bad sequence lemma from the theory of well quasi orders. It may not be entirely unexpected that minimal bad arrays correspond to $\Pi^1_2$-comprehension, since A.~Marcone~\cite{marcone-bad-sequence} has shown that the notion of better quasi order is $\Pi^1_2$-complete. Nevertheless, Marcone himself has formulated the much weaker conjecture that the minimal bad array lemma is unprovable in~$\Pi^1_1\textsf{-CA}_0$ (see Problem~2.19 from~\cite{marcone-survey-old}). The strength of $\Pi^1_2$-comprehension dwarfs the one of $\Pi^1_1$-comprehension, as expressed by M.~Rathjen~\cite{rathjen-icm}. In their analysis of a theorem from topology, C.~Mummert and S.~Simpson~\cite{mummert-simpson} assert that they have found ``the first example of a theorem of core mathematics which is provable in second order arithmetic and implies $\Pi^1_2$~comprehension." So it seems justified to conclude that the minimal bad array lemma has exceptional axiomatic strength.

\section{From minimal bad arrays to $\Pi^1_2$-comprehension}

In this section, we derive $\Pi^1_2$-comprehension from Simpson's version of the minimal bad array lemma, which yields the first implication in Theorem~\ref{thm:main} from the introduction.

We will make crucial use of Marcone's result that the notion of better quasi order is $\Pi^1_2$-complete. While Marcone does not specify a base theory, it is straightforward to check that his proof yields the following.

\begin{proposition}[{\cite{marcone-Pi12-complete}}]\label{prop:Pi12-complete}
Consider a $\Pi^1_2$-formula~$\varphi(n)$ in the language of second order arithmetic, possibly with further parameters. The theory~$\mathsf{ACA_0}$ proves that, for any values of the parameters, there is a family of sets~$Q_n$ and reflexive binary relations~$R_n$ such that $\varphi(n)$ holds precisely if $R_n$ is a better relation on~$Q_n$.
\end{proposition}

Marcone also shows that the relations $R_n$ in the proposition can be turned into quasi orders, based on a construction of M.~Pouzet~\cite{pouzet-better-relation}. However, the latter makes use of $\Pi^1_1$-comprehension in the form of local minimal bad arrays. In any case, we will later construct a relation that is not transitive even when each $R_n$ is. To reduce to quasi orders, we will then use the following result. It strengthens the result by Pouzet, which only covers the case where $\leq'$ is the identity relation. Since the minimal bad array in~(ii) is given, we do not need a stronger base theory. 

\begin{proposition}[$\mathsf{ATR_0}$]\label{prop:pouzet}
When $\leq'$ is a partial ranking of a reflexive relation~$R$ on a set~$Q$, there is a well-founded partial order~$\leq_Q$ on~$Q$ that validates the following:
\begin{enumerate}[label=(\roman*)]
\item We have ${\leq'}\subseteq{\leq_Q}\subseteq R$.
\item For any $\leq_Q$-minimal~$\leq_Q$-bad array $f:[V]^\omega\to Q$, there is a set $W\in[V]^\omega$ such that the restriction of~$f$ to $[W]^\omega$ is $R$-bad.
\end{enumerate}
\end{proposition}
\begin{proof}
Given that partial rankings are well-founded, we find a well order~$(\alpha,\preceq)$ and a bijection $o:Q\to\alpha$ such that $p\leq'q$ entails $o(p)\preceq o(q)$. Let us note that the original proof by Pouzet, in which~$\leq'$ is the identity relation, is based on an arbitrary enumeration with~$\alpha=\omega$. Following Pouzet, we stipulate
\begin{equation*}
p<_Qq\quad:\Leftrightarrow\quad\begin{cases}
\text{we have $o(p)\prec o(q)$ and $p\,R\,q$}\\
\text{and $r<_Q p$ implies $r<_Q q$ for all $r\in Q$,}
\end{cases}
\end{equation*}
which amounts to a recursion over~$o(p)$. It is immediate that this yields a well-founded partial order that is contained in~$R$. Aiming at~(i), we assume $p<'q$. Given that $<'$ is a partial ranking of~$R$, we have $p\,R\,q$ and also $o(p)\prec o(q)$. To~get~$p<_Q q$, we show that $r<_Q p$ entails $r<_Q q$. We use induction on~$o(r)$. Given $r<_Q p$, we have $o(r)\prec o(p)\prec o(q)$ as well as $r\, R\,p\,<'q$. Crucially, the latter entails $r\,R\,q$, due to our definition of partial ranking for relations that are not quasi orders. For $r'<_Q r<_Q p$ we get $r'<_Q q$ by induction hypothesis, as needed to conclude~$r<_Q q$. Let us now consider $f:[V]^\omega\to Q$ as in~(ii). By the open Ramsey theorem, we first find $V'\in[V]^\omega$ with $o(f(X))\preceq o(f(X^-))$ for all $X\in[V']^\omega$ (see Theorem~4.9 in conjunction with Lemma~3.1 of~\cite{marcone-survey-old}). Consider the power\-set~$\mathcal P(Q)$ with the order that is induced by~$\leq_Q$ as in the introduction. We look at the array
\begin{equation*}
h:[V']^\omega\to\mathcal P(Q)\quad\text{with}\quad h(X)=\{r\in Q\,|\,r<_Q f(X)\}.
\end{equation*}
Let us note that the sets $h(X)$ are finite when we have $\alpha=\omega$, as in the original proof by Pouzet. For $W\in[V']^\omega$, the restriction of~$h$ to~$[W]^\omega$ cannot be bad. If~it was, we could choose values $g(X)\in h(X)$ with $g(X)\not\leq_Q r$ for all~$r\in h(X^-)$, by the definition of the order on~$\mathcal P(Q)$. In view of $g(X^-)\in h(X^-)$, the resulting array~$g:[W]^\omega\to Q$ would be $\leq_Q$-bad. But we would also have $g(X)<_Q f(X)$ due to $g(X)\in h(X)$, against the minimality of $f$. Again by the open Ramsey theorem, we thus find a $W\in[V']^\omega$ such that $h(X)\leq_{\mathcal P(Q)} h(X^-)$ holds for all~$X\in[W]^\omega$. Given $r<_Q f(X)$ with $X\in[W]^\omega$, we can now argue that $r\in h(X)$ yields $r\leq_Q r'$ for some $r'\in h(X^-)$, where we have $r'<_Q f(X^-)$ and hence $r<_Q f(X^-)$. We~can conclude that $f(X)\leq_Q f(X^-)$ is equivalent to $f(X)\,R\,f(X^-)$ for~$X\in[W]^\omega$. Given that $f$ is $\leq_Q$-bad, it must thus be $R$-bad on~$[W]^\omega$, as desired.
\end{proof}

The following corollary shows that the minimal bad array lemma for quasi orders entails a version of the same lemma for reflexive relations.

\begin{corollary}[$\mathsf{ATR_0}$]
Assume Simpson's version of the minimal bad array lemma for quasi orders. Consider a partial ranking~$\leq'$ of a reflexive relation~$R$ on a set~$Q$. If~$R$ is no better relation on~$Q$, there is a $Q$-array that is $\leq'$-minimal $R$-bad.
\end{corollary}
\begin{proof}
Consider~${\leq_Q}\subseteq R$ as in the previous proposition. Given that $R$ is no better relation, we get a $\leq_Q$-bad array $f_0$. The minimal bad array lemma for quasi orders yields an array $f'\leq_Q f_0$ that is $\leq_Q$-minimal $\leq_Q$-bad. By the previous proposition, we can restrict~$f'$ to an array~$f$ that is $R$-bad. The latter must indeed be $\leq'$-minimal $R$-bad, as an $R$-bad $g<'f$ would be $\leq_Q$-bad with $g<_Q f\leq_Q f'$.
\end{proof}

Let us note that the previous proof yields $f\leq_Q f_0$ but not necessarily $f\leq' f_0$, so that we obtain an apparently weaker version of the minimal bad array lemma for reflexive relations. A direct proof of the full version is given in the next section. We now deduce the first direction of our main result.

\begin{theorem}[$\mathsf{ATR_0}$]\label{thm:mba-to-Pi12CA}
The principle of $\Pi^1_2$-comprehension follows from Simpson's version of the minimal bad array lemma.
\end{theorem}
\begin{proof}
    Given a $\Pi^1_2$-formula $\varphi(n)$, possibly with parameters, we consider sets $Q_n$ and reflexive relations~$R_n$ as in Proposition~\ref{prop:Pi12-complete}. We need to form the set
    \begin{equation*}
        X=\{n\in\mathbb N\,|\,R_n\text{ is a better relation on }Q_n\}.
    \end{equation*}
    Let $\omega^*$ be the order with underlying set $\mathbb N$ and order relation~${\leq^*}=\{(m,n)\,|\,m\geq n\}$. For each $n$, we consider the disjoint union $Q_n^*=Q_n\cup\omega^*$ with the relation
    \begin{equation*}
    R_n^*=R_n\cup{\leq^*}\cup\{(q,m)\,|\,q\in Q_n\text{ and }m\in\omega^*\}.
    \end{equation*} 
    Now define $Q$ as the set of sequences $\sigma=\langle\sigma_0,\ldots,\sigma_{l(\sigma)-1}\rangle$ with length $l(\sigma)$ and entries $\sigma_i\in Q_i^*$. To obtain a reflexive relation~$R$ on $Q$, we stipulate
    \begin{equation*}
        \sigma\,R\,\tau\quad\Leftrightarrow\quad\text{either $l(\sigma)\geq l(\tau)$ or there is $i<l(\sigma)<l(\tau)$ with $\sigma_i\,R_i^*\tau_i$}.
    \end{equation*}
    Then~$R$ is no better relation. Indeed, we obtain $R_i^*$-bad arrays $f_i:[\mathbb N]^\omega\to Q_i^*$ by setting $f_i(X)=\min(X)\in\omega^*$. An $R$-bad $Q$-array is given by
    \begin{equation*}
     X\mapsto\langle f_0(X),\ldots,f_n(X)\rangle\quad\text{for}\quad n=\min(X).
    \end{equation*}
For $\sigma,\tau\in Q$ we declare that $\sigma\leq'\tau$ holds if we have $l(\sigma)=l(\tau)$ and either $\sigma_i=\tau_i$ or $\sigma_i\in Q_i$ and $\tau_i\in\omega^*$ for each~$i<l(\sigma)$. It is straightforward to see that this yields a partial ranking of~$R$. In view of the previous corollary, we may now consider an array $f:[V]^\omega\to Q$ that is $\leq'$-minimal~$R$-bad. As noted in the introduction, we can identify this array with a map $f:B\to Q$ on a block~$B$ with base~$V$. For $s,t\in B$, one writes $s\vartriangleleft t$ if there is an $X\in[V]^\omega$ with $s\sqsubset X$ and $t\sqsubset X^-$. That $f$ is bad means that $f(s)\,R\,f(t)$ fails for any~$s\vartriangleleft t$. We show
\begin{equation*}
\text{$R_i$ is a better relation on~$Q_i$}\quad\Leftrightarrow\quad\text{$f(s)_i\in\omega^*$ for all $s\in B$ with $i<l(f(s))$}.
\end{equation*}
Here $f(s)_i$ refers to the $i$-th entry of the sequence $f(s)\in Q$. Once the equivalence is established, the aforementioned set~$X$ can be formed by arithmetic comprehension. Let us first assume that $R_i$ is a better relation on~$Q_i$. Towards a contradiction, we assume that $f(s)_i\in Q_i$ holds for some~$s\in B$ with $i<l(f(s))$. Let $B/s$ be the block of all $t\in B$ such that the largest number in~$s$ lies strictly below the smallest number in~$t$. Given $t\in B/s$, we find $r^i\in B$ with $s=r^0\vartriangleleft\ldots\vartriangleleft r^n=t$ (consider some~$X\in[V]^\omega$ with $s\cup t\sqsubset X$). Since $f(r^j)\,R\,f(r^{j+1})$ must fail for all~$j<n$, we obtain $l(f(r^j))<l(f(r^{j+1}))$ and inductively $f(r^j)_i\in Q_i$. Hence we have $i<l(f(s))<l(f(t))$ and $f(t)_i\in Q_i$. But then the array $B/s\ni t\mapsto f(t)_i\in Q_i$ is $R_i$-bad, against our assumption. For the other direction of our equivalence, we assume $f(s)_i\in\omega^*$ holds for all $s\in B$ with $i<l(f(s))$. Pick $r^0\vartriangleleft\ldots\vartriangleleft r^i$ in~$B$. As~above, we in\-duc\-tively get $j\leq l(f(r^j))$ and hence $i<l(f(s))$ for all~$s\in B/r^i$. In other words, if~$V'\in[V]^\omega$ is the base of $B/r^i$, we have $i<l(f(X))$ and $f(X)_i\in\omega^*$ for all~$X\in[V']^\omega$. Towards a contradiction, we now assume that $R_i$ is not better. Let $h:[W]^\omega\to Q_i$ be an $R_i$-bad array. We may assume~$W\subseteq V'$. In order to define $g:[W]^\omega\to Q$, we stipulate that we have $l(g(X))=l(f(X))$ as well as
\begin{equation*}
g(X)_j=\begin{cases}
 h(X) & \text{if $j=i$},\\
 f(X)_j & \text{if $i\neq j<l(g(X))$}.
\end{cases}
\end{equation*}
This yields an $R$-bad array $g<'f$, against minimality.
\end{proof}

\section{From $\Pi^1_2$-comprehension to minimal bad arrays}\label{sect:mba-from-Pi12CA}

In this section, we first show that Simpson's version of the minimal bad array lemma follows from another version that has been isolated by Laver~\cite{laver-min-array}. We then prove that Laver's version is a consequence of $\Pi^1_2$-comprehension. Together with Theorem~\ref{thm:mba-to-Pi12CA}, this completes the proof of Theorem~\ref{thm:main} from the introduction.

So far, we have mostly viewed arrays as continuous functions with domain $[V]^\omega$ for some infinite $V\subseteq\mathbb N$, even though it was agreed that these were represented by functions on blocks. In other words, while different blocks can be used to represent the same continuous function, most previous considerations did not depend on the choice of representative. This is in marked contrast with the following, where the blocks on which arrays are defined form a crucial part of the data. We note that the following definitions are taken from~\cite{laver-min-array}, where $\ledot$ is written $\mathrel{\text{\scalebox{1.2}{$\sqsupseteq$}}}$.

\begin{definition}\label{def:lessdot}
For blocks $B$ and~$C$, we write $B\ledot C$ if we have $\base{B}\subseteq\base{C}$ and each $t\in B$ admits an $s\in C$ with $s\sqsubseteq t$. If this holds and we have $B\not\subseteq C$, then we write $B\lessdot C$. Given arrays $f:B\to Q$ and $g:C\to Q$ as well as a partial order~$\leq'$ on the set~$Q$, we write $f\ledot' g$ if we have $B\ledot C$ as well as $f(t)=g(t)$ for $t\in B\cap C$ and $f(t)<'g(s)$ for $C\ni s\sqsubset t\in B$. If we additionally have $B\lessdot C$, we write $f\lessdot' g$.
\end{definition}

It is immediate that $\ledot$ is transitive on blocks and not hard to derive that $\ledot'$ is transitive on arrays. One can also show that $B\lessdot D$ follows from $B\lessdot C\ledot D$ (but not from $B\ledot C\lessdot D$), which yields the analogous fact for arrays.

\begin{min-bad-arr-lem}[`Laver's version']
Consider a partial ranking $\leq'$ of a reflexive relation~$R$ on a set~$Q$. For any $R$-bad $Q$-array $f_0$ (given as a function on a block), there is an $R$-bad $Q$-array~$f\ledot' f_0$ that admits no $R$-bad $Q$-array $g\lessdot' f$.
\end{min-bad-arr-lem}

The given version does indeed coincide with the one by Laver~\cite{laver-min-array}, except that the latter requires $R$ to be a quasi order and works with barriers at the place of blocks. To see that our formulation in terms of blocks entails the one in terms of barriers, we recall that any block~$B$ admits a barrier $B'\subseteq B$, provably in~$\mathsf{ATR}_0$ (see~\cite{marcone-bad-sequence} and compare~\cite{cholak-RM-wpo} for a related result over the much weaker theory~$\mathsf{WKL}_0$). Given a minimal bad array $f\ledot' f_0$ on a block, we thus find a restriction~$f'\ledot' f$ that is defined on a barrier. By the observations above, we get $f'\ledot' f_0$ while a bad~$g\lessdot' f'$ would validate $g\lessdot' f$, against the assumption that $f$ is minimal. Similarly, the following proposition remains valid when Laver's version of the minimal bad array lemma is formulated in terms of barriers rather than blocks. Once we have proved that $\Pi^1_2$-comprehension entails Laver's version for blocks, the resulting circle of implications will ensure that the versions for blocks and barriers are equivalent. In the same way, we learn that the version for quasi orders is equivalent to the one for general reflexive relations. Let us note that the base theory $\mathsf{RCA_0}$ suffices to get the following result for blocks. We have chosen a stronger base theory to accommodate the aforementioned connection with barriers.

\begin{proposition}[$\mathsf{ATR}_0$]\label{prop:laver-to-simpson}
Laver's version of the minimal bad array lemma entails Simpson's version.
\end{proposition}
\begin{proof}  
Given a bad $Q$-array~$f_0$, Laver's version yields a bad $f\ledot' f_0$ that admits no bad~$g\lessdot' f$.  One readily derives $f\leq' f_0$ according to Definition~\ref{def:minimal-bad}. To conclude that Simpson's version holds, we assume that $g<' f$ is bad and derive a contradiction. Write $f:B\to Q$ and $g:C\to Q$ for blocks $B$ and~$C$, where we have $\bigcup C\subseteq\bigcup B$. Let us note that we do not immediately get the contradictory $g\lessdot' f$, unless each~$t\in C$ admits an $s\in B$ with $s\sqsubset t$. For $X\subseteq\mathbb N$ and $s\in[\mathbb N]^{<\omega}$, let $X/s$ be the set of all numbers in $X$ that are larger than all elements of~$s$. We write $s^\frown n=s\cup\{n\}$ when we have $n\in\mathbb N/s$. Let us now put
\begin{align*}
B_0&=\{s\in B\,|\,\text{there is $t\in C$ with $t\sqsubseteq s\subset\textstyle\bigcup C$}\},\\
C'&=\{t\in C\,|\,\text{there is $s\in B$ with $s\sqsubset t$}\}\cup\{s^\frown n\,|\,s\in B_0\text{ and }n\in{\textstyle\bigcup C}/s\}.
\end{align*}
It is straightforward to see that $C'$ is a block with $\bigcup C'=\bigcup C$. We consider
\begin{equation*}
g':C'\to Q\quad\text{with}\quad g'(t)=\begin{cases}
g(t) & \text{when there is $s\in B$ with $s\sqsubset t$},\\
g(t') & \text{when $t=s^\frown n$ with $C\ni t'\sqsubseteq s$}.
\end{cases}
\end{equation*}
Let us note that $g$ and $g'$ induce the same continuous function from~$[{\bigcup C}]^\omega$ to~$Q$. In particular, $g'$ is bad. To reach the desired contradiction, we verify $g'\lessdot' f$. First note that we have $C'\lessdot B$ as well as $B\cap C'=\emptyset$. Given $B\ni s\sqsubset t\in C'$, we pick a set $X\in[{\bigcup C}]^\omega$ with $t\sqsubset X$. By considering the induced continuous functions, we see that $g<' f$~entails
\begin{equation*}
g'(t)=g'(X)=g(X)<'f(X)=f(s),
\end{equation*}
just as Definition~\ref{def:lessdot} requires.
\end{proof}

In the rest of this section, we show that $\Pi^1_2$-comprehension suffices to carry out the proof of the minimal bad array lemma that has been given by R.~Fra\"iss\'e~\cite{fraisse-theory-relations}.

\begin{definition}
    Let $B$ and $C$ be blocks with $\barrsucc{C}{B}$. We put
    \begin{align*}
        E(C,B,n)&=\{t\in B\,|\,t\subseteq[0,n]\cup \base C\text{ but }t\not\subseteq\base C\}\quad\text{for $n\in\mathbb N$},\\
        M(C,B)&=\{n\in\mathbb N\,|\,\text{there is }t\in B\text{ with }t\subseteq \base C\text{ but }t\not\in C\text{ and }\max t=n\}.
    \end{align*}
    If we have $M(C,B)\neq\emptyset$ or equivalently $C\lessdot B$, we put $m(C,B)=\min M(C,B)$.
\end{definition}

The following basic construction will be needed below. Let us point out that we have $C\cap E(C,B,n)=\emptyset$, as $t\in E(C,B,n)$ entails $t\not\subseteq\bigcup C$ and hence~$t\notin C$. The given proof is very close to the one by Fra{\"i}ss\'e~\cite{fraisse-theory-relations}, except that we work with blocks rather than barriers. We have included it in order to keep our paper self-contained.

\begin{lemma}[$\mathsf{RCA_0}$]
    Assume that $C\lessdot B$ are blocks and that we have $n\leq \minbarr{C}{B}$. Then $D:=C\cup \extbarr{C}{B}{n}$ is a block with $\barrpropsucc{\barrsucc{C}{D}}{B}$ and $\minbarr{C}{B}=\minbarr{D}{B}$.
\end{lemma}
\begin{proof}
    We start by showing that $D$ is a block. It is clear that $\bigcup D\supseteq\bigcup C$ is infinite. Towards a contradiction, we assume that we have $t\sqsubset s$ with $s,t\in D$. Then $s$ and~$t$ are not both from~$C$ and not both from $\extbarr{C}{B}{n}\subseteq B$. For $t\in E(C,B,n)$ we would get $s\supset t\not\subseteq\bigcup C$ and hence~$s\notin C$. In the remaining case, we have $t\in C$ and~$s\in \extbarr{C}{B}{n}$. Due to $\barrpropsucc{C}{B}$, we then find an $s'\in B$ with~$s'\sqsubseteq t\sqsubset s\in B$, against the assumption that~$B$ is a block. Thus $D$ is a $\sqsubseteq$-antichain. Now consider an infinite set $X\subseteq\bigcup D$. We want to prove that $X$ has an initial segment in~$D$. In view of $X\subseteq\base{B}$, we may pick a~$t\in B$ that validates~$t\sqsubset X$. First assume that we have~$t\not\subseteq\bigcup C$. By the definition of $E(C,B,n)$, all elements of $\bigcup D\backslash\bigcup C$ are bounded by~$n$. Hence we have $t\subseteq [0,n]\cup\bigcup C$, so that we get $t\in E(C,B,n)\subseteq D$. Now assume we have $t\subseteq\bigcup C$. If we have $t\in C\subseteq D$, then we are done. So let us assume $t\notin C$. The latter entails $n\leq m(C,B)\leq\max t$. Due to $\bigcup D\subseteq[0,n]\cup\bigcup C$, we get $X\subseteq t\cup\bigcup C\subseteq\bigcup C$. As $C$ is a block, we find an $s\in C\subseteq D$ with $s\sqsubset X$.

    It is straightforward to see that $C\lessdot B$ entails $C\ledot D\lessdot B$. To complete the proof, we show $M(C,B)=M(D,B)$. For $\max t\in M(C,B)$ with $t\in B\backslash C$ but~$t\subseteq\bigcup C$, we obtain $t\notin E(C,B,n)$ and hence $t\notin D$ but $t\subseteq\bigcup D$, which yields $\max t\in M(D,B)$. Conversely, consider $\max s\in M(D,B)$ with $s\in B\backslash D$ but $s\subseteq\bigcup D\subseteq[0,n]\cup\bigcup C$. As $s\not\subseteq\bigcup C$ would entail $s\in E(C,B,n)\subseteq D$, we must have $s\subseteq\bigcup C$. Since we also have $s\not\in D\supseteq C$, we obtain $\max s\in M(C,B)$.
\end{proof}

We conclude our paper with the following result. Together with Theorem~\ref{thm:mba-to-Pi12CA} and Proposition~\ref{prop:laver-to-simpson}, it yields a circle of implications that establishes Theorem~\ref{thm:main}.

\begin{theorem}[$\mathsf{RCA}_0$]\label{thm:pi12-imp-mbal}
    The principle of $\Pi^1_2$-comprehension entails Laver's version of the minimal bad array lemma.
\end{theorem}
\begin{proof}
Recall that $\Pi^1_2$-comprehension is equivalent to the principle that any subset of~$\mathbb N$ is contained in a countable coded~$\beta_2$-model (see Theorems~VII.6.9 and~VII.7.4 of~\cite{simpson09}). Here a $\beta_2$-model is an $\omega$-model for which~$\Pi^1_2$-statements are absolute. With the help of $\beta_2$-models, it is straightforward to formalize the proof of the minimal bad array lemma that is given by Fra\"iss\'e~\cite[7.3.5]{fraisse-theory-relations}. We provide details for the convenience of the reader and also to verify that the argument applies to general reflexive relations rather than just to quasi orders.

For a partial ranking $\leq'$ of a reflexive binary relation~$R$ on a set~$Q$, we consider an \mbox{$R$-bad} array $f_0:B_0\to Q$ on a block. Let us pick a $\beta_2$-model $\mathcal M$ that contains~$f_0$. The statement that a given array~$g$ is minimal $R$-bad has complexity $\Pi^1_2$, as it takes a $\Pi^1_1$-formula to assert that the domain of a potential $f\lessdot' g$ is a block. Assuming that the minimal bad array lemma fails, any $R$-bad array $g\in\mathcal M$ with $g\ledot' f_0$ will thus admit an $R$-bad array $f\lessdot' g$ with $f\in\mathcal M$. We intend to build a sequence of $R$-bad arrays $f_i\ledot'f_0$ with $f_{i+1}\lessdot' f_i$ such that the following additional properties hold, where we write $B_i$ for the block on which~$f_i$ is defined:
\begin{enumerate}[label=(\roman*)]
\item for any $n\in\bigcup B_i$ with $n\leq m(B_{i+1},B_i)$ we have $n\in\bigcup B_{i+1}$,
\item for any $R$-bad $g:C\to Q$ with $g\lessdot' f_i$ we have $m(B_{i+1},B_i)\leq m(C,B_i)$.
\end{enumerate}
The indicated relation between~$f_i$ and $f_{i+1}$ has complexity $\Pi^1_2$, by the same consideration as above. Now $\Pi^1_2$-comprehension entails $\Sigma^1_2$-dependent choice but not $\Pi^1_2$-choice (see Theorem~VII.6.9 and Remark~VII.6.3 of~\cite{simpson09}). We will soon show that any $f_i\in\mathcal M$ admits an $f_{i+1}\in\mathcal M$ with the given properties. Once this is achieved, we can lower the quantifier complexity of our $\Pi^1_2$-condition by evaluating it in the $\beta_2$-model~$\mathcal M$. Then $\Sigma^1_2$-dependent choice is more than sufficient to obtain the desired sequence of arrays~$f_i$. To be more economical, one can pick the smallest indices that represent suitable~$f_i$ in the coded model~$\mathcal M$.

As promised, we now show that each $R$-bad $f_i\ledot' f_0$ in~$\mathcal M$ admits an \mbox{$f_{i+1}\in\mathcal M$} with the properties above. We have already observed that $\mathcal M$ contains an \mbox{$R$-bad} array $f:B\to Q$ with $f\lessdot' f_i$, at least if there is no minimal $R$-bad array below~$f_0$. We pick such an $f$ with $m(B,B_i)$ as small as possible, in the sense that $\mathcal M$ contains no $R$-bad array $g:C\to Q$ with $g\lessdot' f_i$ and $m(C,B_i)<m(B,B_i)$. This minimality condition extends to~$g$ that do not lie in~$\mathcal M$, by $\Pi^1_2$-absoluteness. So~$f$ validates~(ii) but possibly not~(i). To satisfy the latter, we modify~$f$. The challenge is that~(ii) needs to be preserved. We will see that this holds for $f_{i+1}:B_{i+1}\to Q$ with
\begin{align*}
  B_{i+1}&=B \cup \extbarr{B}{B_i}{\minbarr{B}{B_i}},\\
  f_{i+1}(t)&=
  \begin{cases}
    f(t) & \text{if } t\in B,\\
    f_i(t) & \text{if } t\in \extbarr{B}{B_i}{\minbarr{B}{B_i}}.
  \end{cases}
\end{align*}
Let us first note that $f_{i+1}$ lies in~$\mathcal M$ by absoluteness, since it is arithmetically definable from~$f$ and~$f_i$. By the previous lemma, $B_{i+1}$ is a block with $B\ledot B_{i+1}\lessdot B_i$ and $m:=m(B_{i+1},B_i)=m(B,B_i)$. Due to this last equality, (ii)~remains valid. To see that we get $f_{i+1}\lessdot'f_i$, we note that $B_i\ni s\sqsubset t\in B_{i+1}$ forces $t\in B$, since we have $E(B,B_i,m)\subseteq B_i$. It remains to show that $f_{i+1}$ is $R$-bad and validates~(i).

To show that $f_{i+1}$ is $R$-bad, we argue by contradiction. Assume that we have $s,t\in B_{i+1}$ with $s\vartriangleleft t$ and $f_{i+1}(s)\,R\, f_{i+1}(t)$. Since $f$ and $f_i$ are bad, we cannot have $s,t\in B$ or $s,t\in B_i$. Concerning the two other cases, we first assume $s\in B$ and $t\in E(B,B_i,m)\subseteq B_i$. Given that we have $B\lessdot B_i$ and $B\cap E(B,B_i,m)=\emptyset$, we find an $s'\in B_i$ with $s'\sqsubset s$. We can conclude $m\leq \max s'$ due to the definition of $m=m(B,B_i)$. From $t\subseteq[0,m]\cup\bigcup B$ and $s'\subseteq s\subseteq\bigcup B$ we now get $t\subseteq\bigcup B$ and hence $t\notin E(B,B_i,m)$, against our assumption. Let us now assume that we have $s\in E(B,B_i,m)\subseteq B_i$ and $t\in B$. We find a $t'\in B_i$ with $t'\sqsubseteq t$ and thus~$s\vartriangleleft t'$. In view of $f\lessdot' f_i$, we obtain
\begin{equation*}
    f_i(s)=f_{i+1}(s)\,R\,f_{i+1}(t)=f(t)\leq' f_i(t').
\end{equation*}
Crucially, we get $f_i(s)\,R\, f_i(t')$ by our definition of partial rankings for relations that need not be transitive. We now have a contradiction since~$f_i$ is bad.

In order to establish~(i), we consider an $n\in\bigcup B_i$ with $n\leq m=m(B_{i+1},B_i)$. Due to $B\subseteq B_{i+1}$, we are done unless we have $n\notin\bigcup B$. In the latter case, we can find a $t\in B_i$ with $n\in t\subseteq [0,m]\cup\bigcup B$. We thus have $t\in E(B,B_i,m)\subseteq B_{i+1}$ and in particular $n\in\bigcup B_{i+1}$, as desired.

Under the assumption that no minimal $R$-bad array below~$f_0$ exists, we have shown that there is a sequence of $R$-bad arrays $f_{i+1}\lessdot' f_i\ledot' f_0$ that validate~(i) and~(ii) above. With the help of this sequence, we will now show that there is a minimal $R$-bad array below~$f_0$ after all. Let us first prove some properties of the function $p:\mathbb N\to\mathbb N$ with $p(i):=m(B_{i+1},B_i)$.

\begin{claim}
  For every $n\in\mathbb N$ we have $|p^{-1}(\{n\})|\leq 2^n$.
\end{claim}
{\renewcommand{\qedsymbol}{$\diamond$}
\begin{proof}[Proof of the claim]
As there are only $2^n$ sets $t\subseteq[0,n]$ with $\max t=n$, we need only show that the same~$t$ cannot witness both $p(i)=n$ and $p(j)=n$ for $i<j$. Assume $p(i)=n$ is witnessed by $t\in B_i$ with $t\subseteq\bigcup B_{i+1}$ but $t\notin B_{i+1}$. Since $B_{i+1}$ is a block, we find a $t'\in B_{i+1}$ that is $\sqsubset$-comparable with~$t$. In view of $B_{i+1}\lessdot' B_i$, we may consider an $s\in B_i$ with $s\sqsubseteq t'$. This forces $s=t$, so that we get $t\sqsubset t'$. Now if $t$ did witness $p(j)=n$, we would have $t\in B_j$. Due to $B_j\ledot' B_{i+1}$, we could pick an $s'\in B_{i+1}$ with $s'\sqsubseteq t\sqsubset t'\in B_{i+1}$, against the fact that $B_{i+1}$ is a block.
\end{proof}}

We will also need the following.
   
\begin{claim}
  The function $p$ is non-decreasing.
\end{claim}
{\renewcommand{\qedsymbol}{$\diamond$}
\begin{proof}[Proof of the claim]
Write $p(i+1)=\max s$ with $s\in B_{i+1}\backslash B_{i+2}$ and $s\subseteq\bigcup B_{i+2}$. In view of $B_{i+1}\lessdot B_i$, we find a $t\in B_i$ with~$t\sqsubseteq s$. We show how to conclude in either of two cases. First assume $t\sqsubset s$. We then have $t\subseteq s\subseteq\bigcup B_{i+1}$ but $t\notin B_{i+1}$, which yields $\max t\in M(B_{i+1},B_i)$ and hence
\begin{equation*}
p(i)=m(B_{i+1},B_i)\leq\max t<\max s=p(i+1).
\end{equation*}
Now assume that we have $t=s$. This yields $t\in B_i\backslash B_{i+2}$ and $t\subseteq\bigcup B_{i+2}$, so that we obtain $\max t\in M(B_{i+2},B_i)$ and thus
\begin{equation*}
    p(i)=m(B_{i+1},B_i)\leq m(B_{i+2},B_i)\leq\max t\leq\max s=p(i+1),
\end{equation*}
where the first inequality holds due to~(ii) above.
\end{proof}}

Let $H$ be the range of $p$. The first claim above implies that $H$ is infinite. Using the second claim, we establish $H\subseteq\bigcup B_i$ for any~$i\in\mathbb N$. In view of $B_{i+1}\lessdot B_i$, it suffices to show that $p(j)\in H$ is contained in $\bigcup B_i$ for $i\geq j$. First note that we have $p(j)=m(B_{j+1},B_j)=\max t$ for some $t\in B_j$, which yields $p(j)\in t\subseteq\bigcup B_j$. Inductively, we assume $p(j)\in\bigcup B_i$ with~$i\geq j$. Here the latter ensures that we have $p(j)\leq p(i)=m(B_{i+1},B_i)$, so that we get $p(j)\in\bigcup B_{i+1}$ by~(i) above. Let us now consider the set
\begin{equation*}
    C=\{t\in[H]^{<\omega}\,|\,\text{there is $i\in\mathbb N$ with $t\in B_j$ for all~$j\geq i$}\}.
\end{equation*}
Since all the $B_j$ are blocks, it is immediate that $C$ is a $\sqsubset$-antichain. To show that it is a block, we consider an arbitrary set~$X\in[H]^\omega$. Due to $H\subseteq\bigcup B_i$, we may pick an $s_i\in B_i$ with $s_i\sqsubset X$ for each~$i\in\mathbb N$. Since we have $f_{i+1}\lessdot' f_i$, we get $s_{i+1}\sqsubseteq s_i$ and $f_{i+1}(s_{i+1})\leq'f_i(s_i)$, where the latter is an equality precisely for $s_i=s_{i+1}$. Given that $\leq'$ is well-founded, the values $f_i(s_i)$ must indeed stabilize. So we find an $i\in\mathbb N$ such that $j\geq i$ entails $s_i=s_j\in B_j$. This yields $s_i\in C$, as needed to conclude that~$C$ is a block.

Now consider $g:C\to Q$ with $g(t)=f_i(t)$ for $t\in B_i\cap C$, which is well-defined by the previous paragraph. It is straightforward to show that $g$ is $R$-bad and that we have $g\ledot' f_i$ for any~$i\in\mathbb N$ (in fact we get $g\lessdot' f_i$). To complete the proof, we show that $g$ is minimal. Aiming at a contradiction, we assume that $h:D\to Q$ is $R$-bad with~$h\lessdot' g$. By the paragraph after Definition~\ref{def:lessdot}, we get $h\lessdot' f_i$ for any~$i\in\mathbb N$. Write $m(D,C)=\max t$ with $t\in C$ and $t\subseteq\bigcup D$ but $t\notin D$. When~$j\in\mathbb N$ is sufficiently large, we have both $t\in B_j$ and $m(D,C)<p(j)$. We get
\begin{equation*}
    m(D,B_j)\leq\max t=m(D,C)<p(j)=m(B_{j+1},B_j),
\end{equation*}
which yields a contradiction with property~(ii) from above.
\end{proof}

\bibliographystyle{amsplain}
\bibliography{Min-bad-array_Pi12-CA}

\end{document}